\documentclass[12pt]{article}
\usepackage{amsmath,amssymb,mathrsfs,wasysym}
\usepackage{amsthm}
\usepackage{enumerate}
\usepackage{graphicx}
\date{}
\newtheorem{theorem}{Theorem}
\newtheorem{prop}{Proposition}
\newtheorem{lemma}{Lemma}
\newtheorem{cor}{Corollary}

\theoremstyle{definition}

\newtheorem{exmp}{Example}

\begin{document}
\author{Edyta Bartnicka}
  \title{Tops of graphs of projective codes}
  \maketitle

  \begin{abstract}
Let $\Gamma_k(V)$ be the Grassmann graph whose vertex set  ${\mathcal G}_{k}(V)$  is formed by all $k$-dimensional subspaces of an $n$-dimensional vector space $V$ over the finite field $F_q$ consisting of $q$ elements. Denote by  $\Pi[n,k]_q$ the subgraph of $\Gamma_k(V)$ formed by projective  codes.  We give a complete description of cliques $\langle U]^{\Pi}_{k}$ of  $\Pi[n,k]_q$ consisting of all $k$-dimensional projective codes contained in a fixed  $(k+1)$-dimensional subspace of $V$. We show when and in how many lines of  ${\mathcal G}_{k}(V)$ they are contained. Next we prove that $\langle U]^{\Pi}_{k}$ is a maximal clique of $\Pi[n,k]_q$ exactly if it
  is contained in at most one line of  ${\mathcal G}_{k}(V)$.
\end{abstract}



\section{Introduction}

The Grassmann graph  $\Gamma_k(V)$ formed by the set ${\mathcal G}_{k}(V)$ of all $k$-dimensional subspaces of an $n$-dimensional vector space $V$ over the $q$-element field $F_q$ can be identified with the graph of all linear  codes $[n, k]_q$, where two distinct linear codes $[n, k]_q$ are connected by an edge if 
their intersection is $(k-1)$-dimensional. We assume that  $1<k<n-1$ since the Grassmann graph is complete for $k=1$ and $k=n-1$. The subgraphs of the Grassmann graph  have been investigated e.g., in
\cite{CGK, KP2, Pankov2023}.
Maximal cliques in these graphs  are considered in \cite{
KP4} for example.

We focus on the subgraph $\Pi[n,k]_q$ of the Grassmann graph formed by the set $\Pi(n,k)_q$ of all projective codes $[n,k]_{q}$, i.e., linear codes such that columns in their generator matrices are non-zero and mutually non-proportional. The graph of projective codes have been discussed in \cite{KP3}.
In this article, we investigate cliques $\langle U]^{\Pi}_{k}$ of  $\Pi[n,k]_q$ consisting of all $k$-dimensional projective codes contained in a fixed  $(k+1)$-dimensional subspace $U$ of $V$. We are particularly interested in those $\langle U]^{\Pi}_{k}$  which are maximal cliques of $\Pi[n,k]_q$.

Maximal cliques  of $\Pi[n,k]_q$ are intersections of maximal cliques of $\Gamma_k(V)$ with the set $\Pi(n,k)_q$. If $1<k<n-1$, then any maximal clique of the Grassmann graph is precisely one of the following types: stars $[S\rangle_k$ (made up of all $k$-dimensional subspaces of $V$ containing a fixed $(k-1)$-dimensional subspace $S$) or tops (formed by all $k$-dimensional subspaces of $V$ contained in a fixed $(k+1)$-dimensional subspace $U$). Any  intersection of a star or a top with the set $\Pi(n, k)_q$ is a clique (if it is non-empty), however it need not be a maximal clique.  We call  such intersection a star or a top  of $\Pi[n,k]_q$, respectively, if it is a maximal clique of $\Pi[n,k]_q$. We refer to \cite{B} for the description of the stars of $\Pi[n,k]_q$. Here we  provide the detailed characterization of the tops  of $\Pi[n, k]_q$. 
The starting point of our investigation is introducing the subspace $Y$ of $(k+1)$-dimensional vector space $V^{k+1}$ over the finite field $F_q$. This subspace is generated by all vectors which are not linear combinations of precisely two columns of the generator matrix for a fixed projective code $[n,k+1]_{q}=U$.  Next, we prove that the set $\langle U]^{\Pi}_{k}$ is contained in a line of ${\mathcal G}_{k}(V)$  in the case of $\dim Y\leqslant2$  (Theorem \ref{liney}). Finally, we show that $\langle U]^{\Pi}_{k}$ is a top of $\Pi[n,k]_q$  when it is not contained in any line of ${\mathcal G}_{k}(V)$ or it is a star of $\Pi[n,k]_q$ simultaneously (Theorem \ref{projective top}). 
 Examples of maximal and non-maximal cliques $\langle U]^{\Pi}_{k}$
 are also given.

\section{Preliminaries}
Let $F_q$ be the finite field consisting of $q$ elements  and let  $V=F_{q}^{n}$ be the  $n$-dimensional vector space  over $F_q$. Consider  the {\it Grassmann graph} $\Gamma_{k}(V)$ whose the set of vertices is Grassmannian  ${\mathcal G}_{k}(V)$ consisting of all  $k$-dimensional subspaces of $V$.
This is the simple and connected graph whose  vertices are joined by an edge if their intersection is $(k-1)$-dimensional.
Any element of ${\mathcal G}_{k}(V)$ is interpreted as a {\it linear code} $[n,k]_{q}$.  In practice, only non-degenerate linear codes are useful. A linear code is said to be  {\it non-degenerate} if
its generator matrices do not contain zero columns.  If
 generator matrices for linear code contain neither zero nor proportional columns, then such code is called  {\it projective}.

Let $S, U$ be subspaces of $V$ such that $S\subset U\subseteq V$ and $\dim S<k<\dim U$. The set $K$ of all $k$-dimensional subspaces  of $V$ such that $S\subset K\subset U$ is denoted by $[S,U]_{k}$. We write  $\langle U]_{k}$ and $[S\rangle_{k}$ in the  cases when $S=0$ or $U=V$, respectively. If $$\dim S=k-1\;\mbox{ and }\;\dim U=k+1,$$ then $[S,U]_{k}$ is called a {\it line}.

Recall that a clique in
an undirected graph is a  set of  vertices with any two distinct
elements  joined by an edge. 
 A clique  is said to be {\it maximal} if it is not
properly contained in any clique.
We suppose  that $1<k<n-1$  throughout the paper (since the Grassmann graph is complete if $k=1$ or $k=n-1$) and then every  maximal clique of $\Gamma_{k}(V)$ is of one of the following types:\begin{enumerate}
\item[$\bullet$]
a {\it star}  $[S\rangle_{k}$ consisting of all $k$-dimensional subspaces of $V$ containing a fixed subspace $S\in {\mathcal G}_{k-1}(V)$,
or
\item[$\bullet$] a {\it top}  $\langle U]_{k}$ consisting of all $k$-dimensional subspaces of $V$ contained in a fixed  subspace $U\in {\mathcal G}_{k+1}(V)$.\end{enumerate}
The intersection of $[S\rangle_{k}$ and $\langle U]_{k}$  is non-empty exactly if
it is a line of ${\mathcal G}_k(V)$ and then it contains precisely $q+1$ elements. The intersection of two distinct maximal cliques of the  same type  contains at most one vertex.

Write $\Pi[n, k]_q$  for the restriction of the  Grassmann graph to the set $\Pi(n, k)_q$ of all projective codes. This is the simple  graph  which is connected when $q\geq \binom{n}{2}$ (see \cite[Theorem 1]{KP3}).
Every clique of $\Pi[n, k]_q$ is a  clique of $\Gamma_{k}(V)$ and  every maximal clique of $\Pi[n, k]_q$ is
 the intersection of $\Pi(n,k)_{q}$ with a maximal clique of $\Gamma_{k}(V)$.
The intersections $[S\rangle_{k}\cap{\Pi}(n,k)_q$ and $\langle U]_{k}\cap{\Pi}(n,k)_q$ are denoted by $[S\rangle^{\Pi}_{k}$ and $\langle U]^{\Pi}_{k}$, respectively.
They both
may be empty or  non-maximal cliques of $\Pi[n, k]_q$. We say that $[S\rangle^{\Pi}_{k}$ is a {\it star} of $\Pi[n, k]_q$ or $\langle U]^{\Pi}_{k}$ is a {\it top} of $\Pi[n, k]_q$ only if they are maximal cliques of $\Pi[n, k]_q$.

Recall that $V$ contains precisely
$$\genfrac{[}{]}{0pt}{}{n}{k}_q=
\frac{(q^n-1)(q^{n-1}-1)\cdots(q^{n-k+1}-1)}{(q-1)(q^2-1)\cdots(q^k-1)}$$
$k$-dimensional subspaces. Hence the number of $1$-dimensional subspaces of $V$, which is also the number of hyperplanes in $V$, is equal
$$\genfrac{[}{]}{0pt}{}{n}{1}_q=\genfrac{[}{]}{0pt}{}{n}{n-1}_q=
[n]_{q}=\frac{q^n-1}{q-1}.$$
Furthermore, the total number of $(k+1)$-dimensional subspaces of $V$ containing a fixed $k$-dimensional subspace of $V$ is  $[n-k]_q$.

\section{Tops of graphs of projective codes}

It is obvious that if a $(k+1)$-dimensional subspace $U$ of $V$ does not belong to $\Pi(n,k+1)_q$, then $\langle U]^{\Pi}_{k}$ is empty. That is why throughout this paper we shall only study $U\in \Pi(n,k+1)_q$. A generator matrix of $U$, i.e., a matrix whose rows $v_{1},\dots,v_{k+1}$  form a basis for $U$, will be denoted by $M$.

It is clear that $\langle U]^{\Pi}_{k}$ is the subset of the set $\langle U]^{c}_{k}$ of all $k$-dimensional non-degenerate linear codes contained in $U$. According to \cite{KP2} $\langle U]^{c}_{k}$ is linked with the subset $W$ of $V^{k+1}=F_{q}^{k+1}$ formed by all non-zero vectors $w=(w_{1},\dots,w_{k+1})$ whose coordinates are not proportional to coordinates of any column of $M$. Namely,  $\langle U]^{c}_{k}$ consists of all $C(w), w\in W$, where  for every $w\in W$ $C(w)$ is defined  as the set of all
vectors $\sum^{k+1}_{i=1}a_{i}v_{i}$ with scalars $a_{1},\dots, a_{k+1}\in F_q$ satisfying the equality
$$\sum^{k+1}_{i=1}a_{i}w_{i}=0.$$

Let now $Y\subseteq W$ be the set of all elements $w\in W$ such that $w\neq \alpha l_i+\beta l_j$ for all $\alpha, \beta\in F_q\setminus\{0\}$ and all columns $l_i, l_j$ of $M$. We will write $C(y)$ instead of $C(w)$ in the case of $y\in Y\subseteq W$.
 \begin{lemma}\label{ukpi}
   The set $\langle U]^{\Pi}_{k}$ consists of all $C(y), y\in Y$.
 \end{lemma}
 \begin{proof}
   Let $M=\left[
\begin{array}{cccc}
v_1\\
\vdots\\
v_{k+1}\\
\end{array}
\right]=\left[
\begin{array}{cccc}
v_{11}& \cdots &v_{1n}\\
\vdots&\ddots&\vdots\\
v_{(k+1)1}& \cdots &v_{(k+1)n}\\
\end{array}
\right]$, $w=\left[
\begin{array}{cccc}
w_1\\
\vdots\\
w_{k+1}\\
\end{array}
\right]$ and \begin{align*}
C&=\left[\begin{array}{c}
a_{11}v_{1}+ \cdots +a_{1(k+1)}v_{k+1}\\
\vdots\\
a_{k1}v_{1}+ \cdots +a_{k(k+1)}v_{k+1}\\
\end{array}
\right]=\\
              &=\left[\begin{array}{ccc}
a_{11}v_{11}+ \cdots +a_{1(k+1)}v_{(k+1)1}& \cdots &a_{11}v_{1n}+ \cdots +a_{1(k+1)}v_{(k+1)n}\\
\vdots&\ddots&\vdots\\
a_{k1}v_{11}+ \cdots +a_{k(k+1)}v_{(k+1)1}& \cdots &a_{k1}v_{1n}+ \cdots +a_{k(k+1)}v_{(k+1)n}\\
\end{array}
\right]      \end{align*} be a generator matrix for $C(w)$. Equivalently, $
a_{s1}w_{1}+ \cdots +a_{s(k+1)}w_{k+1}=0$ for all $s=1,2, \dots, k$.

\noindent
Suppose that there exist $\alpha, \beta\in F_q\setminus\{0\}$ and the $i$-th, the  $j$-th columns $l_i, l_j$ of $M$ such that $w=\alpha l_i+\beta l_j$, that is,  $$\left[
\begin{array}{cccc}
w_1\\
\vdots\\
w_{k+1}\\
\end{array}
\right]=\left[
\begin{array}{cccc}
\alpha v_{1i}+\beta v_{1j}\\
\vdots\\
\alpha v_{(k+1)i}+\beta v_{(k+1)j}\\
\end{array}
\right].$$
This means that for all $s=1,2, \dots, k$
\begin{align*}
&a_{s1}(\alpha v_{1i}+\beta v_{1j})+ \cdots +a_{s(k+1)}(\alpha v_{(k+1)i}+\beta v_{(k+1)j})=\\
&=\alpha( a_{s1}v_{1i}+\cdots+ a_{s(k+1)}v_{(k+1)i})        +     \beta (a_{s1}v_{1j}+ \cdots +a_{s(k+1)}v_{(k+1)j})=
0\end{align*}
or equivalently $$a_{s1}v_{1i}+\cdots+ a_{s(k+1)}v_{(k+1)i})=-\alpha^{-1}\beta(a_{s1}v_{1j}+ \cdots +a_{s(k+1)}v_{(k+1)j}).$$
This shows that the $i$-th and the  $j$-th columns of $C(w)$ are proportional exactly if $w$ is a linear combination of the $i$-th and the  $j$-th columns of $M$.
 \end{proof}

 \begin{prop}\label{dimCy}
 Let $Y\neq \emptyset$ and let $Y^\perp$ be the orthogonal complement of the subspace $\langle Y\rangle$ of $V^{k+1}$. Then
   $$\dim Y^\perp=\dim\bigcap_{y\in Y}C(y).$$
\end{prop}
\begin{proof}
Suppose that $a=\left[
\begin{array}{cccc}
a_1\\
\vdots\\
a_{k+1}\\
\end{array}
\right] \in Y^\perp$, i.e., $\sum^{k+1}_{i=1}a_{i}y_{i}=0$ holds for all\linebreak $y=\left[
\begin{array}{cccc}
y_1\\
\vdots\\
y_{k+1}\\
\end{array}
\right]\in Y$. This is equivalent to saying that $\sum^{k+1}_{i=1}a_{i}y_{i}\in \bigcap_{y\in Y}C(y)$.
\end{proof}
\begin{theorem}\label{liney}
  The set $\langle U]^{\Pi}_{k}$ is contained in a line of  ${\mathcal G}_{k}(V)$ if and only if $\dim Y\leqslant2$.
  Additionally, $\langle U]^{\Pi}_{k}$ is contained in
  \begin{enumerate}[a)]
    \item\label{dim0} all lines of ${\mathcal G}_{k}(V)$  if and only if $\dim \langle Y\rangle=0$.
    \item\label{dim1} precisely $[k]_q[n-k]_q$ lines of ${\mathcal G}_{k}(V)$  if and only if $\dim \langle Y\rangle=1$.
    \item\label{dim2} precisely one line of ${\mathcal G}_{k}(V)$  if and only if $\dim \langle Y\rangle=2$. The line is of the form $[\bigcap_{y\in Y}C(y), U]_k$.
  \end{enumerate}
\end{theorem}
  \begin{proof}
    Assume first that $\dim \langle Y\rangle>2$. By Proposition \ref{dimCy} this is equivalent to saying that $\dim \bigcap_{y\in Y}C(y)<k-1$. This means that there is no any $(k-1)$-dimensional subspace $S$ of $V$ such that  $\langle U]^{\Pi}_{k}\subseteq [S, U]_k$. Thus $\langle U]^{\Pi}_{k}$ is not contained in any line of  ${\mathcal G}_{k}(V)$.

    \noindent
    We will now consider other cases of $\dim \langle Y\rangle$.
    \begin{enumerate}[a)]
    \item $\dim \langle Y\rangle=0\Leftrightarrow\langle U]^{\Pi}_{k}=\emptyset$. Clearly, $\langle U]^{\Pi}_{k}$ is contained in
   all lines of ${\mathcal G}_{k}(V)$.
    \item $\dim \langle Y\rangle=1\Leftrightarrow|\langle U]^{\Pi}_{k}|=1$ and there exist $[k]_q$ $(k-1)$-dimensional subspaces $S$ of $C(y)\in \langle U]^{\Pi}_{k}$. Moreover, there exist $[n-k]_q$ $(k+1)$-dimensional subspaces $U'$ of $V$ containing $C(y)$. Therefore there exist precisely $[k]_q[n-k]_q$ lines $[S, U']_k$ containing $\langle U]^{\Pi}_{k}$.
    \item $\dim \langle Y\rangle=2\Leftrightarrow\dim \bigcap_{y\in Y}C(y)=k-1$ by Proposition \ref{dimCy}. In other words,  $\bigcap_{y\in Y}C(y)$ is the only $(k-1)$-dimensional subspace of $V$ contained in any  $C(y)\in \langle U]^{\Pi}_{k}$. Furthermore, $|\langle U]^{\Pi}_{k}|\geqslant2$ and hence $\langle U]^{\Pi}_{k}\nsubseteq\langle U']^{\Pi}_{k}$ for any $(k+1)$-dimensional subspaces $U'\neq U$. Consequently, $[\bigcap_{y\in Y}C(y), U]_k$ is the only line of ${\mathcal G}_{k}(V)$ containing $\langle U]^{\Pi}_{k}$.
  \end{enumerate}
  \end{proof}

  \begin{theorem}\label{projective top}
  The set $\langle U]^{\Pi}_{k}$ is a top of $\Pi[n,k]_{q}$ if and only if it
  is not contained in any line of  ${\mathcal G}_{k}(V)$ or $\bigcap_{y\in Y}C(y)$ is a $(k-1)$-dimensional subspace of $V$ such that $\langle U]^{\Pi}_{k}=[\bigcap_{y\in Y}C(y)\rangle^{\Pi}_{k}$.
\end{theorem}
\begin{proof}
Assume first that $\langle U]^{\Pi}_{k}$
  is contained in more than one line of  ${\mathcal G}_{k}(V)$.  On account of Theorem \ref{liney},  $\langle U]^{\Pi}_{k}$ is an empty set or it contains precisely one element $C(y)$. In the latter case,  there exists a $(k-1)$-dimensional subspace $S$ of $C(y)$ and in view of \cite[Proposition 1]{B} $\langle U]^{\Pi}_{k}$ is a proper subset of $[S\rangle^{\Pi}_{k}$. Therefore $\langle U]^{\Pi}_{k}$ is not a maximal clique of $\Pi[n,k]_{q}$.

  Let now $\langle U]^{\Pi}_{k}$
  be contained in  precisely one line of ${\mathcal G}_{k}(V)$, then by Theorem \ref{liney}, $\dim \langle Y\rangle=2$ and the line is $[\bigcap_{y\in Y}C(y), U]_k$.  In consequence, $\langle U]^{\Pi}_{k}\nsubseteq\langle U']^{\Pi}_{k}$ if $U'\neq U$ (since the intersection of two distinct tops of ${\mathcal G}_{k}(V)$ contains no more than one element) and $\langle U]^{\Pi}_{k}\nsubseteq[S\rangle^{\Pi}_{k}$ if $S\neq \bigcap_{y\in Y}C(y)$. If there exists a $k$-dimensional subspace of $V$ containing $\bigcap_{y\in Y}C(y)$ but not contained in $U$, then  $\langle U]^{\Pi}_{k}\subset[\bigcap_{y\in Y}C(y)\rangle^{\Pi}_{k}$ and is not a maximal clique. Otherwise, $\langle U]^{\Pi}_{k}=[\bigcap_{y\in Y}C(y)\rangle^{\Pi}_{k}$ is a top and a star of $\Pi[n,k]_{q}$.

 If $\langle U]^{\Pi}_{k}$ is not contained in any line of  ${\mathcal G}_{k}(V)$, then we get immediately that $\langle U]^{\Pi}_{k}$ is not contained in any star of $\Pi(n,k)_{q}$. According to Theorem \ref{liney}, $|\langle U]^{\Pi}_{k}|>2$, and so $\langle U]^{\Pi}_{k}$ is not contained in any other top of $\Pi[n,k]_{q}$. Consequently,
$\langle U]^{\Pi}_{k}$ is a maximal clique of $\Pi[n,k]_{q}$.



\end{proof}

\begin{exmp}
  Consider three projective codes $U=[10, 5]_2$ whose generator matrices differ precisely one column and the set $Y$ in each case has different number of elements.
  \begin{enumerate}[a)]
    \item Let $M=\left[\begin{matrix}
                                                        v_1 \\
                                                        v_2 \\
                                                        v_3 \\
                                                        v_4 \\
                                                        v_5
                                                      \end{matrix}\right]=\left[\begin{matrix}
                   1 & 0 & 0 & 0 & 0 & 1 & 0 & 1 & 1 & 0 \\
                   0& 1  & 0 & 0 & 0 & 1 & 0 & 0 & 0 & 1 \\
                   0 & 0 & 1 & 0 & 0 & 0 & 1 & 1 & 0& 0\\
                   0 & 0 & 0 & 1 & 0 & 0 & 1 & 0 & 1 & 0 \\
                   0 & 0 & 0 & 0 & 1 & 0 & 0 & 0 & 1 & 1
                 \end{matrix}\right].$ Then $Y=\left\{\left[\begin{matrix}
                                                        1 \\
                                                        1 \\
                                                        1 \\
                                                        1 \\
                                                        1
                                                      \end{matrix}\right]\right\}$ and $\langle U]^{\Pi}_{4}=\left\{C(y)\right\}$, where $C(y)=\langle v_1+v_2, v_2+v_3, v_3+v_4, v_4+v_5\rangle.$ Thus we can write a generator matrix for $C(y)$ as
                                                      $$\left[\begin{matrix}
                   1 & 1 & 0 & 0 & 0 & 0 & 0 & 1 & 1 & 1 \\
                   0& 1  & 1 & 0 & 0 & 1 & 1 & 1 & 0 & 1 \\
                   0 & 0 & 1 & 1 & 0 & 0 & 0 & 1 & 1& 0\\
                   0 & 0 & 0& 1 & 1 & 0 & 1 & 0 & 0 & 1
                 \end{matrix}\right].$$ It is easy to see that the matrix
                  $$\left[\begin{matrix}
                   1 & 1 & 0 & 0 & 0 & 0 & 0 & 1 & 1 & 1 \\
                   0& 1  & 1 & 0 & 0 & 1 & 1 & 1 & 0 & 1 \\
                   0 & 0 & 1 & 1 & 0 & 0 & 0 & 1 & 1& 0\\
                   0 & 0 & 0& 1 & 1 & 0 & 1 & 1 & 1 & 1
                 \end{matrix}\right]$$ is an example of generator matrix for a projective code $X=[10, 4]_2$ such that $X\nsubseteq U$ and $\dim\left(X\cap C(y)\right)=k-1$. Therefore $\langle U]^{\Pi}_{4}$ is not a top of $\Pi[10,4]_{2}$.
    \item Choose now $M=\left[\begin{matrix}
                                                        v_1 \\
                                                        v_2 \\
                                                        v_3 \\
                                                        v_4 \\
                                                        v_5
                                                      \end{matrix}\right]=\left[\begin{matrix}
                   1 & 0 & 0 & 0 & 0 & 1 & 0 & 1 & 1 & 0 \\
                   0& 1  & 0 & 0 & 0 & 1 & 0 & 0 & 0 & 1 \\
                   0 & 0 & 1 & 0 & 0 & 0 & 1 & 1 & 0& 1\\
                   0 & 0 & 0 & 1 & 0 & 0 & 1 & 0 & 1 & 0 \\
                   0 & 0 & 0 & 0 & 1 & 0 & 0 & 0 & 1 & 0
                 \end{matrix}\right].$ In this case $Y=\left\{\left[\begin{matrix}
                                                        0 \\
                                                        1 \\
                                                        1 \\
                                                        1 \\
                                                        1
                                                      \end{matrix}\right],\left[\begin{matrix}
                                                        1 \\
                                                        1 \\
                                                        1 \\
                                                        0 \\
                                                        1
                                                      \end{matrix}\right]\right\}$ and $\langle U]^{\Pi}_{4}$ consists of two elements $C(y_1)=\langle v_2+v_3, v_2+v_5, v_1+v_2+v_4, v_1\rangle $ and $C(y_2)=\langle v_2+v_3, v_2+v_5, v_1+v_2+v_4, v_4\rangle$.

                 Then $C(y_1)\cap C(y_2)=\langle v_2+v_3, v_2+v_5, v_1+v_2+v_4\rangle $ is generated by
                                                      $$\left[\begin{matrix}
                   0 & 1 & 1 & 0 & 0 & 1 & 1 & 1 & 0 & 0 \\
                   0& 1  & 0 & 0 & 1 & 1 & 0 & 0 & 1 & 1 \\
                   1 & 1 & 0 & 1 & 0 & 0 & 1 & 1 & 0& 1
                 \end{matrix}\right]$$  and the matrix
                 $$\left[\begin{matrix}
                   0 & 1 & 1 & 0 & 0 & 1 & 1 & 1 & 0 & 0 \\
                   0& 1  & 0 & 0 & 1 & 1 & 0 & 0 & 1 & 1 \\
                   1 & 1 & 0 & 1 & 0 & 0 & 1 & 1 & 0& 1\\
                   0 & 0 & 0& 1 & 0 & 0 & 0 & 1 & 1 & 0
                 \end{matrix}\right]$$
                  generates a projective code $X=[10, 4]_2$ which contains $(k-1)$-dimensional subspace $C(y_1)\cap C(y_2)$ and is not contained in $U$. Consequently,
                   $\langle U]^{\Pi}_{k}\subset[C(y_1)\cap C(y_2)\rangle^{\Pi}_{k}$ and is not a maximal clique.
    \item Consider $M=\left[\begin{matrix}
                                                        v_1 \\
                                                        v_2 \\
                                                        v_3 \\
                                                        v_4 \\
                                                        v_5
                                                      \end{matrix}\right]=\left[\begin{matrix}
                   1 & 0 & 0 & 0 & 0 & 1 & 0 & 1 & 1 & 0 \\
                   0& 1  & 0 & 0 & 0 & 1 & 0 & 0 & 0 & 1 \\
                   0 & 0 & 1 & 0 & 0 & 0 & 1 & 1 & 0& 0\\
                   0 & 0 & 0 & 1 & 0 & 0 & 1 & 0 & 1 & 1 \\
                   0 & 0 & 0 & 0 & 1 & 0 & 0 & 0 & 1 & 0
                 \end{matrix}\right].$ Here we get $Y$ consists of four elements:
                 $$y_1=\left[\begin{matrix}
                                                        0 \\
                                                        1 \\
                                                        1 \\
                                                        0 \\
                                                        1
                                                      \end{matrix}\right], y_2=\left[\begin{matrix}
                                                        0 \\
                                                        1 \\
                                                        1 \\
                                                        1 \\
                                                        1
                                                      \end{matrix}\right], y_3=\left[\begin{matrix}
                                                        1 \\
                                                        1 \\
                                                        1 \\
                                                        0 \\
                                                        1
                                                      \end{matrix}\right], y_4=\left[\begin{matrix}
                                                        1 \\
                                                        1 \\
                                                        1 \\
                                                        1 \\
                                                        1
                                                      \end{matrix}\right].$$ We see immediately that $y_1+y_2+y_3=y_4$ and $\dim\langle Y\rangle=3$.

                                                      In consequence,  $\langle U]^{\Pi}_{4}$ is made up by  $C(y_1)=\langle v_2+v_5, v_3+v_5, v_1, v_4\rangle$, $C(y_2)=\langle v_2+v_5, v_3+v_5, v_1, v_1+v_2+v_4\rangle$ and $C(y_3)=\langle v_2+v_5, v_3+v_5, v_4, v_1+v_2+v_4\rangle$ and so it is not contained in any $\langle U']^{\Pi}_{4}$, where $U\neq U'$.  Since $\dim\bigcap_{y_i\in Y}C(y_i)=\dim\langle v_2+v_5, v_3+v_5\rangle=2$, there is no any $(k-1)$-dimensional subspace $S$ such that $\langle U]^{\Pi}_{4}\subseteq [S\rangle^{\Pi}_{4}$ and so $\langle U]^{\Pi}_{4}$ is a top of $\Pi[10,4]_{2}$.
  \end{enumerate}
\end{exmp}

\begin{cor}
  The set $\langle U]^{\Pi}_{k}$ is a top and a star of $\Pi[n,k]_{q}$ if and only if the only line of ${\mathcal G}_{k}(V)$ containing $\langle U]^{\Pi}_{k}$
  is $[\bigcap_{y\in Y}C(y), U]_k$ and $\langle U]^{\Pi}_{k}=[\bigcap_{y\in Y}C(y)\rangle^{\Pi}_{k}$.
\end{cor}
\begin{exmp}
  Let $M=\left[\begin{matrix}
               1 & 0 & 0 & 0 & 1& 1\\
               0 & 1 & 0 & 0 & 1 & 0 \\
               0 & 0 & 1 & 0 & 0 & 1 \\
               0 & 0 & 0 & 1 & 1 & 1
             \end{matrix}\right]$ be a generator matrix for a projective code $U=[6, 4]_2$. Then $Y=\left\{\left[\begin{matrix}
                                                        0 \\
                                                        1 \\
                                                        1 \\
                                                        1
                                                      \end{matrix}\right],\left[\begin{matrix}
                                                        1 \\
                                                        1 \\
                                                        1 \\
                                                        0
                                                      \end{matrix}\right]\right\}$ and $\langle U]^{\Pi}_{3}=\left\{C(y_1), C(y_2)\right\}$, where $C(y_1)=\langle v_1+v_2+v_4, v_2+v_3, v_1\rangle, C(y_2)=\langle v_1+v_2+v_4, v_2+v_3, v_4\rangle.$ We can write a generator matrix
                                                      for $C(y_1)\cap C(y_2)=\langle v_1+v_2+v_4, v_2+v_3\rangle$ as
                                                       $$M_S=\left[\begin{matrix}
               1 & 1 & 0 & 1 & 1& 0\\
               0 & 1 & 1 & 0 & 1 & 1
             \end{matrix}\right].$$
     A        generator matrix           for any projective code containing $C(y_1)\cap C(y_2)$ can be obtained from $M_S$ by adding a certain row vector in $2^3$ ways since $M_S$       consists of three different pairs of the same columns. This is also the number of vectors in $C(y_1)$ or $C(y_2)$ which are not elements of $C(y_1)\cap C(y_2)$. Hence all  projective codes $[6, 3]_2$ containing $C(y_1)\cap C(y_2)$ are contained in $\langle U]^{\Pi}_{3}$. This means that $\langle U]^{\Pi}_{3}=[C(y_1)\cap C(y_2)\rangle^{\Pi}_{3}$ is a top and a star of $\Pi[6,3]_{2}$.

\end{exmp}

\footnotesize Edyta Bartnicka\\
Institute of Information Technology,
Faculty of Applied Informatics
and Mathematics,
Warsaw University of Life Sciences - SGGW,
Nowoursynowska 166 St.,
02-787 Warsaw,
Poland\\
{\tt edyta\_bartnicka@sggw.edu.pl}
\end{document}